\documentclass[twoside,11pt,reqno]{amsart}
\usepackage{amsmath,amssymb,amscd,mathrsfs,amscd}
\usepackage{graphics,verbatim}
\usepackage{todonotes}
\usepackage{graphicx}
\usepackage{hyperref}
\usetikzlibrary{arrows}

\newcommand{\losemi}{{\otimes \kern -.78em \ltimes}}
\newcommand{\rosemi}{{\otimes \kern -.78em \rtimes}}

\newcommand{\Hom}{\ensuremath{\operatorname{Hom}}}
\newcommand{\End}{\ensuremath{\operatorname{End}}}

\newcommand{\0}{\bar 0}
\newcommand{\1}{\bar 1}
\newcommand{\Z}{\mathbb{Z}}
\newcommand{\C}{\mathbb{C}}

\newcommand{\gl}{\ensuremath{\mathfrak{gl}}}

\newcommand{\fg}{\ensuremath{\mathfrak{g}}}

\newcommand{\fp}{\ensuremath{\mathfrak{p}}}

\newcommand{\TT}{\mathbf{T}}

\renewcommand{\bar}{\overline}
\newcommand{\BB}{\mathcal{B}}
\newcommand{\leftdownJ}{\rotatebox[origin=c]{180}{\textup{\texttt{J}}}} 
\newcommand{\rightdownJ}{\reflectbox{\rotatebox[origin=c]{180}{\textup{\texttt{J}}}}} 
\newcommand{\leftupJ}{\reflectbox{\textup{\texttt{J}}}}
\newcommand{\rightupJ}{\mathtt{J}}
\newcommand{\Id}{\operatorname{Id}}
\newcommand{\ideal}{\mathcal{I}}

\newtheorem{theorem}{Theorem}[subsection]

\makeatletter\let\c@fact\c@theorem\makeatother

\makeatletter\let\c@note\c@theorem\makeatother

\newtheorem{lemma}{Lemma}[subsection]
\makeatletter\let\c@lemma\c@theorem\makeatother

\makeatletter\let\c@alg\c@theorem\makeatother

\makeatletter\let\c@remark\c@theorem\makeatother

\makeatletter\let\c@example\c@theorem\makeatother

\newtheorem{prop}{Proposition}[subsection]
\makeatletter\let\c@prop\c@theorem\makeatother

\makeatletter\let\c@conj\c@theorem\makeatother

\makeatletter\let\c@cor\c@theorem\makeatother

\newtheorem{defn}{Definition}[subsection]
\makeatletter\let\c@defn\c@theorem\makeatother

\numberwithin{equation}{subsection}

\usepackage[capitalise]{cleveref}
\crefname{theorem}{Theorem}{Theorems}
\crefname{fact}{Fact}{Facts}
\crefname{note}{Note}{Notes}
\crefname{lemma}{Lemma}{Lemmas}
\crefname{alg}{Algorithm}{Algorithms}
\crefname{remark}{Remark}{Remarks}
\crefname{example}{Example}{Examples}
\crefname{prop}{Proposition}{Propositions}
\crefname{conj}{Conjecture}{Conjectures}
\crefname{cor}{Corollary}{Corollaries}
\crefname{defn}{Definition}{Definitions}
\crefname{equation}{}{}

\begin{document}
\title{The Marked Brauer Category}

\author{Jonathan R. Kujawa}
\address{Department of Mathematics \\
          University of Oklahoma \\
          Norman, OK 73019}
\thanks{Research of the first author was partially supported by NSF grant
DMS-1160763}\
\email{kujawa@math.ou.edu}
\author{Benjiman C. Tharp}
\address{Department of Mathematics \\
          University of Oklahoma \\
          Norman, OK 73019}
\email{btharp@math.ou.edu}
\date{\today}
\subjclass[2000]{Primary 17B56, 17B10; Secondary 13A50}

\begin{abstract} We introduce the marked Brauer algebra and the marked Brauer category.  These generalize the analogous constructions for the ordinary Brauer algebra to the setting of a homogeneous bilinear form on a $\Z_{2}$-graded vector space.   We classify the simple modules of the marked Brauer algebra over any field of characteristic not two.  Under suitable assumptions we show that the marked Brauer algebra is in Schur-Weyl duality with the Lie superalgebra, $\fg$, of linear maps which leave the bilinear form invariant.  We also provide a classification of the indecomposable summands of the tensor powers of the natural representation for $\fg$ under those same assumptions.   In particular, our results generalize Moon's work on the Lie superalgebra of type $\mathfrak{p}(n)$ and provide a unifying conceptual explanation for his results. 
\end{abstract}

\maketitle

\section{Introduction}\label{} 

\subsection{}  Let $k$ be a field and let $V$ be an $n$-dimensional $k$-vector space.  We have a natural left action by $G=\operatorname{GL}(V)$ on the $r$-fold tensor product $V^{\otimes r}$ for any $r \geq 0$ (Where $V^{\otimes 0}$ is the trivial module by convention).  On the other hand the symmetric group on $r$ letters, $\Sigma_{r}$, has a right action on $V^{\otimes r}$ given by place permutation.  These two actions clearly commute with each other and so we have an algebra map 
\[
k\Sigma_{r} \to \End_{kG}\left(V^{\otimes r} \right).
\]  Schur proved for $k=\C$ that this map is surjective in general and injective whenever the dimension of $V$ is sufficently large \cite{Schur}.  Such mutually centralizing actions have been an active area of research ever since.

In particular, Brauer considered the case when $G$ is either the orthogonal or symplectic group \cite{Brauer}.  That is, the elements of $\operatorname{GL}(V)$ that leave invariant a nondegenerate bilinear form which is symmetric or antisymmetric, respectively.  For any $\delta \in k$ Brauer defined what is now known as the Brauer algebra, $B_{r}(\delta)$, via a diagrammatic basis.  He proved that if we set $\delta$ to be the dimension of $V$ (resp.\ the negative of the dimension of $V$) when the bilinear form is symmetric (resp.\ antisymmetric), we have a map
\begin{equation*}
B_{r}(\delta) \to \End_{kG}\left(V^{\otimes r} \right).
\end{equation*}

For at least twenty years Brauer's work has been known to admit a pleasing common generalization in the $\Z_{2}$-graded setting  (e.g.\ see \cite{Benkart, BLR}).   Namely, let $V=V_{\0} \oplus V_{\1}$ be a $\Z_{2}$-graded vector space with a nondegenerate, symmetric, even bilinear form on $V$ (see \cref{SS:LSAprelim} for the precise definition of these terms).  In the graded setting it is more convenient to work with Lie superalgebras. If we write $\fg$ for the Lie superalgebra of all endomorphisms of $V$ which leave the bilinear form invariant, then $\fg$ is the orthosymplectic Lie superalgebra $\mathfrak{osp}(V)$.  It is known that if we set $\delta$ equal to the superdimension of $V$ (i.e.\ the difference of the dimensions of $V_{\0}$ and $V_{\1}$), then again we have a map
\begin{equation*}
B_{r}(\delta) \to \End_{\fg}\left(V^{\otimes r} \right).
\end{equation*}

 If the dimension of $V_{\1}$ (resp.\ the dimension of $V_{\0}$) is zero, then $\fg$ is the orthogonal (resp.\ symplectic) Lie algebra.  Thus the $\Z_{2}$-graded setting provides a true unifying generalization of Brauer's work.  Moreover, the fact that $\delta$ is the superdimension provides a conceptual explanation for the negative dimension which appeared in Brauer's original work  

\subsection{}   We now assume $V$ is a $\Z_{2}$-graded vector space which admits a nondegenerate, symmetric \emph{odd} bilinear form.  One can again consider the Lie superalgebra, $\fg$, of all endomorphisms which leave the bilinear form invariant.  In this case $\fg$ is the Lie superalgebra $\fp (n)$ \cite{Kac}.  The representation theory of this Lie superalgebra is still poorly understood.  

It is natural to hope to define an analogue of the Brauer algebra in this case.  In \cite{Moon} Moon defined an algebra, $A_{r}$, by generators and relations and proved that there is a map
\[
A_{r} \to \End_{\mathfrak{p}(n)}\left(V^{\otimes r} \right)
\] and that this map is an isomorphism whenever $n\geq r$. Moon used Bergman's Diamond Lemma to prove that the dimension of $A_{r}$ is equal to the Brauer algebra $B_{r}(0)$ and observed that the generators and relations of $A_{r}$ bears a striking resemblance to those of $B_{r}(0)$.  However, he was unable to provide a conceptual explanation for these similarities. 

As an application of his construction Moon was able to use careful calculations in $A_{2}$ and $A_{3}$ to determine the indecomposable summands of $V^{\otimes 2}$ and $V^{\otimes 3}$, respectively.  However, studying an algebra given by generators and relations is not particularly easy and Moon's algebra has received little attention to date.  The representations of $\mathfrak{p}(n)$ are mysterious and an effective description of $A_{r}$ and its representation theory would be valuable tool.

\subsection{}  Our goal in this paper is to provide a diagrammatically defined algebra which provides a common generalization of the algebras of Brauer and Moon.  Our main results are as follows.  Fix a field $k$ of characteristic not two and fix $\varepsilon \in \{\pm 1 \}$.  If $\varepsilon=1$, then fix $\delta \in k$ and if $\varepsilon=-1$, then set $\delta=0$.  All of the following is done uniformly for all such $\varepsilon$ and $\delta$.  In \cref{SS:MarkedBrauerDiagrams} we define, for every $r,s \geq 0$ such that $r+s$ is even, the space of \emph{marked Brauer diagrams} $B_{r,s}(\delta, \varepsilon)$ spanned by natural generalizations of Brauer's original diagrams.  
 
In \cref{L:basis} we prove that $B_{r,s}(\delta, \varepsilon)$ has a basis consisting of certain marked Brauer diagrams which we call \emph{standard}.  In particular, this implies it has dimension $(r+s-1)!!$.  In \cref{SS:TwoOperations} we define a bilinear operation 
\[
B_{r,s}(\delta, \varepsilon) \otimes B_{s,t}(\delta, \varepsilon) \to B_{r,t}(\delta, \varepsilon)
\] given by the vertical stacking of diagrams.  This is a natural generalization of Brauer's original operation on diagrams.

It is convenient to package together the above data into the combinatorially defined \emph{marked Brauer category}, $\BB (\delta, \varepsilon)$. The objects in this category are the nonnegative integers, $\{[r] \mid r \in \Z_{\geq 0} \}$.  The morphisms are defined to be 
\[ 
\Hom_{\BB (\delta, \varepsilon)}([r],[s]) = B_{r,s}(\delta, \varepsilon)
\] and the composition of morphisms is given by the stacking operation described above.  This category is a strict $k$-linear tensor category and in \cref{T:generatorsandrelations} we provide a description of this category in terms of generators and relations. 

By definition the \emph{marked Brauer algebra} $B_{r}(\delta, \varepsilon)= B_{r,r}(\delta, \varepsilon)$ is the endomorphism algebra in this category of the object $[r]$. It is clear from the definition that $B_{r}(\delta, 1)$ is isomorphic to Brauer's original algebra $B_{r}(\delta)$.  When $\varepsilon=-1$, \cref{T:generatorsandrelations} provides a description of $B_{r}(0,-1)$ in terms of generators and relations and we immediately see that $B_{r}(0,-1)$ is isomorphic to Moon's algebra $A_{r}$.  In particular, we recover Moon's dimension formula as an easy corollary of the aforementioned basis theorem.  Furthermore the marked Brauer algebra provides a natural unifying setting which explains the similarities between $B_{r}(0,1)$ and $A_{r}\cong B_{r}(0,-1)$ observed by Moon.

In \cref{S:MarkedBrauerAlgebra} we initiate a study of the marked Brauer algebra using tools introduced by Koenig and Xi in their study of the ordinary Brauer algebra as a cellular algebra \cite{KX1, KX2}.  In particular, in \cref{T:inflation} we show that the marked Brauer algebra is an iterated inflation (without anti-involution) and use this in \cref{T:simplesoftheMBA} to classify the simple modules of the marked Brauer algebra over any field of characteristic not equal to two.  The main result of \cite{KX1} is that an algebra is cellular if and only if it is an iterated inflation of the ground field with a compatible involutive anti-homomorphism.   

A natural question at this stage is if the marked Brauer algebra is also a cellular algebra.  The marked Brauer algebra admits two involutive anti-homomorphisms: one which correspond to the rotation of diagrams by 180 degrees and one which corresponds to reflecting across the horizontal axis.   However when $\varepsilon=-1$ these maps are not compatible with the iterated inflation construction and, indeed, in this case the marked Brauer algebra is not a cellular algebra.  However, it still satisfies a graded version of the Koenig-Xi condition and we expect that many of the techniques of cellular algebras should generalize.

In \cref{S:LSAs} we come full circle and relate the marked Brauer algebra to representations of Lie superalgebras.  Let $V=V_{\0}\oplus V_{\1}$ be a $\Z_{2}$-graded vector space and set $\delta = \dim V_{\0} - \dim V_{\1}$.  Let $(-,-)$ be a nondegenerate symmetric bilinear form which is homogeneous and of degree $\bar {b} \in \Z_{2}$.  Set $\varepsilon = (-1)^{\bar{b}}$.  Let $\fg$ be the Lie subsuperalgebra of $\mathfrak{gl}(V)$ consisting of all linear maps which leave the bilinear form invariant.  Set $T_{\fg}(V)$ to be the full subcategory of $\fg$-modules consisting of objects $V^{\otimes r}$ for $r \geq 0$.

Using the presentation of the marked Brauer category given in \cref{T:generatorsandrelations} we obtain functors of strict $k$-linear tensor categories 
\begin{align*}
F: &\BB(\delta, \varepsilon) \to T_{\fg}(V).
\end{align*}  This functor naturally extends to the Karoubi envelope of the additive envelope of these categories. In the case when $k=\C$ this functor is known to be full and faithful on objects whenever the dimension of $V$ is sufficiently large compared to the tensor power.  In particular, our classification of the simple modules for the marked Brauer algebra immediately provides a complete classification of the indecomposable summands of $V^{\otimes r}$ under these conditions.  As a corollary we recover Moon's classification results for $V^{\otimes 2}$ and $V^{\otimes 3}$.

\subsection{}
 
The results of this paper raise a number of questions.  The ordinary Brauer algebra has a rich representation theory (e.g.\ \cite{CD,CDM, CDM2, Li}) and we expect an equally interesting theory for the marked Brauer algebra.  As we mentioned in the previous section, the marked Brauer algebra is no longer a cellular algebra when $\varepsilon=-1$ but it does admit an involutive anti-homomorphism which satisfies a graded version of the cellular condition of \cite{KX1}.  This raises the natural question of developing a theory of algebras which are graded cellular (as opposed to cellular algebras which are graded).  We expect that many of the results for cellular algebras would carry over to this more general setting and would allow one to consider cell modules, etc.\ for the marked Brauer algebra.

In a somewhat different direction Comes-Wilson \cite{CW} study a walled Brauer algebra analogue of the marked Brauer category first introduced by Deligne \cite{Del} and use it to provide a complete description of the indecomposable summands of the mixed tensor spaces for the Lie superalgebra $\mathfrak{gl}(m|n)$.  Similar results along with a grading, Koszulity, and a relation to a certain generalization of Khovanov's arc algebra are obtained by Brundan-Stroppel \cite{BS}.  It would be interesting to develop similar results for the marked Brauer category.  In particular, when $\varepsilon=-1$ we expect to obtain new information about the poorly understood representation theory of the Lie superalgebra $\mathfrak{p}(n)$.  Also noteworthy is the recent work by Brundan-Comes-Nash-Reynolds and Brundan-Reynolds on using cyclotomic quotients of the affine walled Brauer category to obtain certain tensor product categorifications of $\mathfrak{sl}_{\infty}$-modules \cite{Brundan, BCNR, BR}.  It would be interesting to investigate similar questions for the marked Brauer category.

As we discussed above, over any field of odd characteristic there is a functor between the marked Brauer category and the category of tensor powers of the natural representation for the Lie superalgebra $\fg$.  Even over $\mathbb{C}$ it is not fully settled when this functor is faithful and full; that is, when the algebra homomorphism from the marked Brauer algebra to the endomorphisms of a tensor power is injective and surjective.  It is also interesting to describe the kernel and image of this map as in \cite{LZ0}.   This is equivalent to obtaining the First and Second Fundamental Theorems of Invariant Theory in this setting.  Another natural question is if there exists a marked version of the Birman-Murakami-Wenzl algebra.  Preliminary results of Grantcharov and Guay suggest the existence of a quantum group corresponding to the Lie superalgebra $\mathfrak{p}(n)$ \cite{GG} and it is natural to expect a deformation of the marked Brauer algebra which plays the role of the BMW algebra in this setting.  See \cite{BGJKW, JK} for similar results for $\mathfrak{q}(n)$. 

Finally, we would like to mention that during the writeup of this work we learned of recent work by Serganova in which she also provides a diagrammatic description of Moon's algebra $A_{r}$ \cite{Ser, SerPC}.  It would be interesting to compare her results to ours.

\subsection{Acknowledgments}  It is our pleasure to acknowledge helpful conversations with Georgia Benkart, Dimitar Grantcharov, Nicolas Guay, Catharina Stroppel, and Vera Serganova at various stages of this project.

\section{Marked Brauer Diagrams}\label{S:MarkedBrauerDiagrams}

\subsection{}\label{SS:prelims}  Let $k$ be a field with characteristic different from two.  Unless otherwise stated, all vector spaces considered in this paper will be finite dimensional $k$-vector spaces.  A \emph{superspace} is a $\Z_{2}$-graded vector space, $V= V_{\0} \oplus V_{\1}$.  Given a superspace $V$ and a homogenous element $v \in V$, we write $\bar{v} \in \Z_{2}$ for its degree.  For short we call an element of $V$ \emph{even} (respectively, \emph{odd}) if $\bar{v}=\0$ (respectively, $\bar{v}=\1$).  We view $k$ itself as a superspace concentrated in degree $\0$.  Given a superspace $V$ we say the \emph{dimension} of $V$ is $m|n$ to record that the dimension of $V_{\0}$ is $m$ and the dimension of $V_{\1}$ is $n$.  In particular, the dimension of $V$ as a vector space is $m+n$ and the \emph{superdimension} of $V$ is, by definition, $m-n$.

If $V$ and $W$ are superspaces, then $V \otimes W$ is naturally a superspace where we grade pure tensors by the formula $\bar{{v\otimes w}} = \bar{v} + \bar{w}$ for all homogenous $v\in V$ and $w \in W$.  Similarly, the space of $k$-linear maps, $\Hom_{k}(V,W)$ is naturally $\Z_{2}$-graded by declaring that a linear map $f:V \to W$ is of degree $r \in \Z_{2}$ if $f(V_{s}) \subseteq W_{r+s}$ for all $s \in \Z_{2}$.

\subsection{}\label{SS:MarkedBrauerDiagrams}  Fix $r,s \in \Z_{\geq 0}$, $\varepsilon \in \left\{\pm 1 \right\}$, and $\delta \in k$ with $\delta=0$ if $\varepsilon=-1$.   A \emph{Brauer diagram} for the pair $(r,s)$ is a partition of the set $\{1,2 \dotsc , r, 1', 2' \dotsc , s' \}$ into two-element subsets.  We follow Brauer's original convention in \cite{Brauer} of representing a Brauer diagram pictorially as a graph in the plane as follows.  Given a Brauer diagram we draw $r$ vertices labelled by $1, \dotsc , r$ in a horizontal row along the $y=1$ line, $s$ vertices labelled by $1', \dotsc , s'$ in a horizontal row along the $y=0$ line, and an edge without self-intersections connecting each pair of vertices which appear in the same two element subset.  To avoid clutter we usually leave the labeling of the vertices implicit.  We consider two Brauer diagrams equivalent if they represent the same set partition.  That is, if their corresponding pictures can be obtained from one another by a regular isotopy in the plane. 

Let us now introduce terminology for Brauer diagrams which will be useful in what follows.  We call an edge which connects two vertices in the top row a \emph{cup}, an edge which connects two vertices in the bottom row a \emph{cap}, and an edge which connects a vertex in the top row to a vertex in the bottom row a \emph{through string}.  We call an imaginary horizontal line through a Brauer diagram a \emph{line of latitude}.

A \emph{marked Brauer diagram} is a Brauer diagram which has been decorated according to the following rules:

\begin{itemize}
\item each cup is decorated with a diamond which we call a \emph{bead};
\item each cap is decorated with an \emph{arrow}; 
\item if we use the word \emph{marking} to refer to either a bead or an arrow, then we require that markings never lie on the same line of latitude.
\end{itemize}
The following picture is an example of a marked Brauer diagram:

\begin{figure}[h]
\centering
\includegraphics{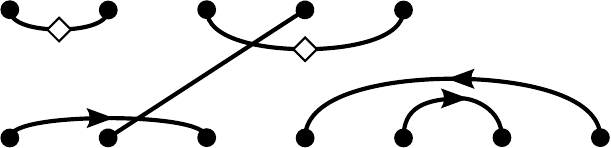}.
\end{figure}

We consider two marked Brauer diagrams equivalent if their corresponding graphs can be obtained from one another by a regular isotopy which never puts two markings simultaneously onto the same line of latitude.  We call two markings \emph{adjacent} if there are no markings which lie on a line of latitude strictly between them.    

We can now introduce the space of marked Brauer diagrams.  In what follows, if a marking is allowed to be either an arrow or bead, then we denote it by a star and we use the convention that different stars in the same diagram may denote different kinds of markings.

\begin{defn}\label{D:markedBrauerDiagrams} Fix $r,s \in \Z_{\geq 0}$, $\varepsilon \in \left\{\pm 1 \right\}$, and $\delta \in k$ with $\delta=0$ if $\varepsilon=-1$.   Define $B_{r,s}(\delta, \varepsilon)$ to be the vector space spanned by the marked Brauer diagrams subject to the following local relations:

\begin{figure}[h]
\centering
\includegraphics[scale=1]{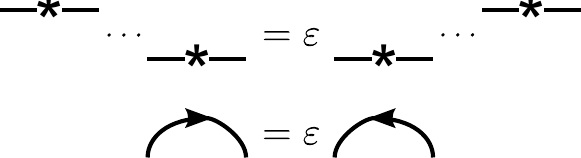}.
\end{figure}
\end{defn}

The first relation indicates that an $\varepsilon$ is introduced whenever two adjacent markings trade lines of latitude within a possibly larger marked Brauer diagram which is otherwise unchanged.  Similarly, the second relation indicates that an $\varepsilon$ is introduced whenever an arrow switches direction within a  possibly larger marked Brauer diagram which is otherwise unchanged.

Observe that when $r+s$ is an odd integer there are no marked Brauer diagrams. In this case we declare $B_{r,s}(\delta, \varepsilon) =0$.  We also note that if $\varepsilon=1$, then we have a vector space isomorphism given by forgetting the markings from $B_{r,s}(\delta, 1)$ to the space of ordinary Brauer diagrams, $B_{r,s}(\delta)$,  studied, for example, in \cite{LZ1}.

\begin{defn}\label{R:grading}  We define a $\Z_{2}$-grading on $B_{r,s}(\delta, \varepsilon)$ by setting the degree of a marked Brauer diagram, $D$, as follows.  If $\varepsilon=1$, then we declare $D$ to always be even.  If $\varepsilon=-1$, then we declare the degree of $D$ to be the number of markings on $D$ considered modulo two. 
\end{defn}

\subsection{}
We put a total order on the markings of a marked Brauer diagram by putting a total order on the cups and caps as follows.  We declare all cups to be before all caps in the ordering.  We order the cups in increasing order by their left vertices as you go left to right.  We order the caps in decreasing order by their left vertices as you go left to right.  The total order on the markings in a marked Brauer diagram is then given by ordering them according to the cup or cap on which they appear.  We call a marked Brauer diagram \emph{standard} if every arrow points to the right and if every marking in the diagram is below all other markings which occur earlier in the total ordering.  An example of a standard marked Brauer diagram is given below.
 
\begin{figure}[h]
\centering
\includegraphics[scale=1]{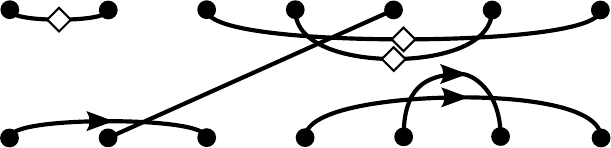}
\end{figure}

\begin{theorem}\label{L:basis}  The set of standard marked Brauer diagrams provides a homogeneous basis for $B_{r,s}(\delta, \varepsilon)$.  In particular, if we set $n=r+s$, then the dimension of $B_{r,s}(\delta, \varepsilon)$ is $\frac{(2n)!}{2^{n}n!}=(2n-1)!!$.

\end{theorem}

\begin{proof} The relations given in \cref{D:markedBrauerDiagrams} imply that every marked Brauer diagram in $B_{r,s}(\delta, \varepsilon)$ is a standard diagram up to scaling by a power of $\varepsilon$ and, hence, spanning is clear.  It is straightforward to count the standard marked Brauer diagrams and see that it is the claimed number. 

The space of ordinary Brauer diagrams, $B_{r,s}(\delta)$, is well known to have the desired dimension so it suffices to define a surjective vector space map from $B_{r,s}(\delta, \varepsilon)$ to $B_{r,s}(\delta)$. We do this as follows. Given a marking, $m$, on a marked Brauer diagram $D$, let $d(m)$ be the number of markings on $D$ which are strictly before $m$ in the ordering and are below $m$.  Given a marked Brauer diagram $D$ we define 
\[
d(D)= (\text{number of left pointing arrows in $D$})+\sum_{m} d(m),
\]
where the sum is over all markings in $D$.  For example, note that $D$ is standard if and only if $d(D)=0$. We then have a well defined surjective linear map $B_{r,s}(\delta, \varepsilon) \to B_{r,s}(\delta)$ defined on marked Brauer diagrams by $D \mapsto \varepsilon^{d(D)}\tilde{D}$, where $\tilde{D}$ is the unmarked Brauer diagram obtained from $D$ by forgetting the markings.
\end{proof}

\subsection{}\label{SS:TwoOperations}

There are two combinatorially defined operations on marked Brauer diagrams.  The easier to define is given by horizontal concatenation of diagrams with the general rule given by linearity.  It will be immediate from the definition that it is associative and preserves the $\Z_{2}$-grading.  For nonnegative integers $r,s,t,u$ the \emph{tensor product} is the bilinear map 
\[
B_{r,s}(\delta, \varepsilon) \otimes B_{t,u}(\delta, \varepsilon) \to B_{r+t, s+u}(\delta, \varepsilon),
\] 
\[
(D_{1},  D_{2}) \mapsto D_{1} \otimes D_{2},
\] where $D_{1} \otimes D_{2}$ is the marked Brauer diagram obtained by concatenating $D_{1}$ to the left of $D_{2}$ and where all markings of $D_{1}$ are taken to be above all markings in $D_{2}$. Pictorially, we can represent $D_{1} \otimes D_{2}$ as:

\begin{figure}[h]
\centering
\includegraphics{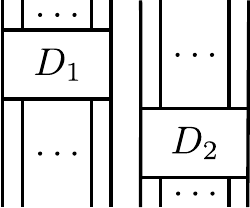}.
\end{figure}

The second operation is given by vertical concatenation of diagrams with the general rule given by linearity.  It is again straightforward to see that this operation is associative and preserves the $\Z_{2}$-grading.   For nonnegative integers $r,s,t$ \emph{composition} is the linear map, 
\[
B_{r,s}(\delta, \varepsilon) \times B_{s,t}(\delta, \varepsilon) \to B_{r,t}(\delta, \varepsilon),
\]
\[
(D_{1},  D_{2}) \mapsto D_{1} \circ D_{2}.
\]  We define composition on diagrams by vertically concatenating $D_{1}$ above $D_{2}$ and identifying the bottom row of vertices of $D_{1}$ with the top row of vertices of $D_{2}$.  A marked Brauer diagram is obtained from this diagram by simplifying using the following two rules:
\begin{enumerate}
\item If the concatenation of the two diagrams creates $m$ closed loops, then the loops are deleted and the resulting diagram is scaled by $\delta^{m}$.
\item If adjacent markings appear on the same edge, then they should be cancelled in pairs according the following local cancellation rules:

\begin{figure}[h]
\centering
\includegraphics{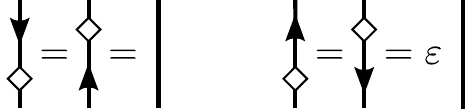}.
\end{figure}
\end{enumerate}
See Figure~1 for an illustrative example of the composition of two marked Brauer diagrams. Note that for composition to be associative when $\varepsilon=-1$ we must have $\delta=0$.  This can already be seen in $B_{2,2}(\delta, -1)$.

\begin{figure}[h]
\centering
\includegraphics[scale=.95]{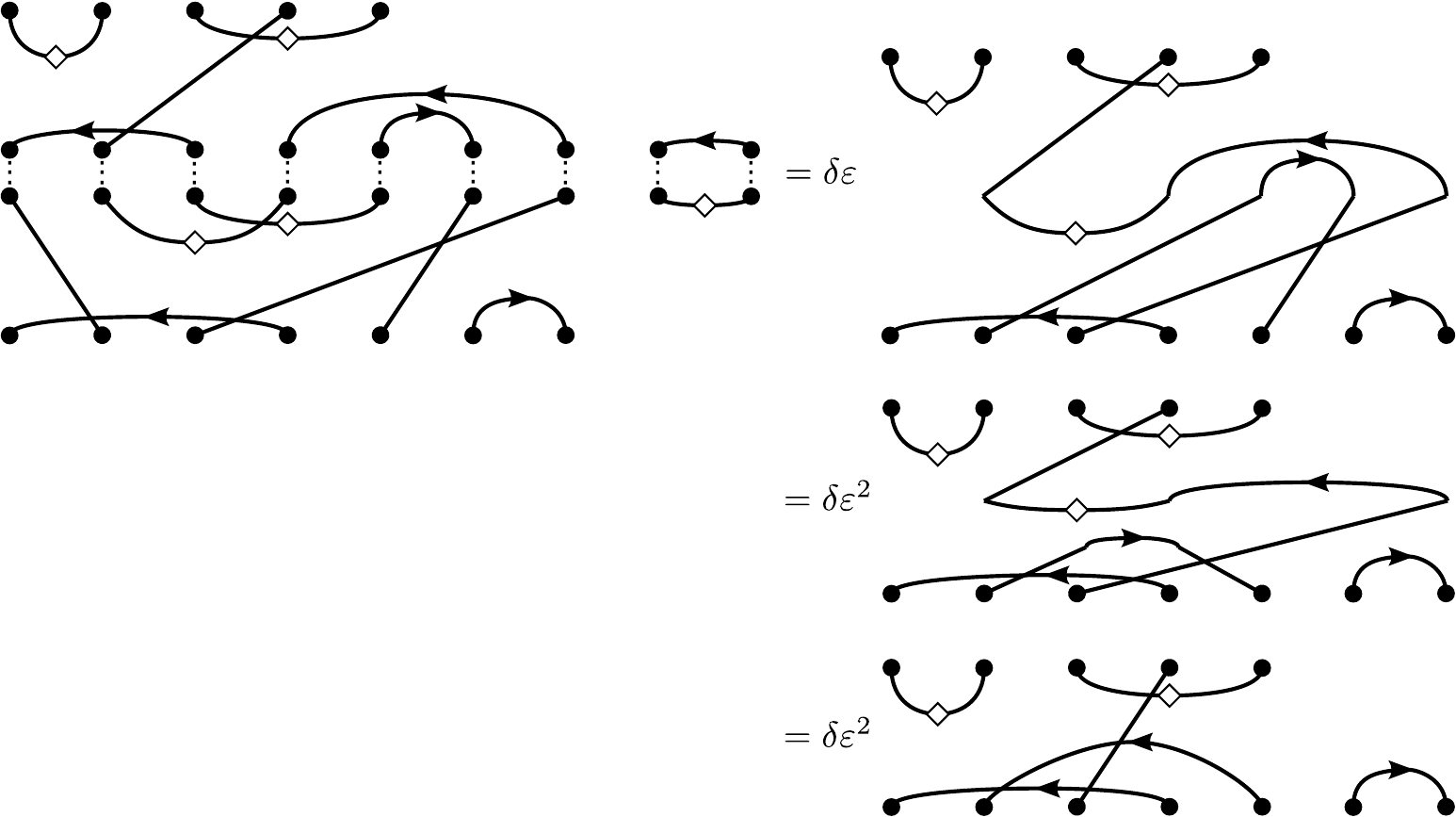}
\caption{An example of the composition operation.}
\end{figure}

\section{The Marked Brauer Category}\label{S:MarkedBrauerCategory}
\subsection{}  
As with tangles and ordinary Brauer diagrams \cite{Kassel,LZ1, Turaev}, it is convenient to describe the structure of marked Brauer diagrams using the language of categories. 

\begin{defn}\label{D:MarkedBrauerCat}   We define the \emph{marked Brauer category} $\BB(\delta, \varepsilon)$ to be the category with objects 
\[
\left\{[a] \mid  a=0,1,2,3,\dotsc \right\}
\] and morphisms 
\[
\Hom_{\BB (\delta, \varepsilon)}([a],[b]) : = B_{a,b}(\delta, \varepsilon).
\]  The composition of morphisms is given by the composition operation defined in \cref{SS:TwoOperations}. 

\end{defn}
From \cref{R:grading} we obtain a natural $\Z_{2}$-grading on each $\Hom$-space in $\BB(\delta, \varepsilon)$. The interested reader can easily check that $\BB(\delta, \varepsilon)$ is enriched over finite-dimensional $k$-superspaces.  We also observe that $\BB (\delta, \varepsilon)$ is a strict tensor category in the sense of \cite{Kassel}.  The tensor product bi-functor is defined on objects by $[a]\otimes[b] = [a+b]$ and on morphisms via the tensor product operation introduced in \cref{SS:TwoOperations}.  Furthermore it is a symmetric tensor category with braiding $\beta_{a,b}:[a] \otimes [b]\to [b] \otimes [a]$ given by the following diagram:

\begin{figure}[h]
\centering
\includegraphics[scale=1]{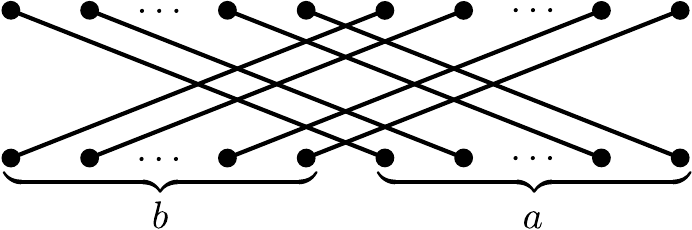}.
\end{figure}

\subsection{}\label{SS:generatorsandrelations}

The marked Brauer category has a presentation by generators and relations as a strict $k$-linear tensor category.   See \cite{Kassel, Turaev} for similar results for the tangle category and \cite{LZ1} for the ordinary Brauer category.

\begin{theorem}\label{T:generatorsandrelations}  Let $I, X, \cup$, and $\cap$ denote, respectively, the marked Brauer diagrams: 

\begin{figure}[h]
\centering
\includegraphics[scale=1]{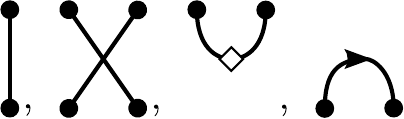}.
\end{figure}

The category $\BB(\delta, \varepsilon)$ is generated as a strict $k$-linear tensor category by $I,X, \cup$, and $\cap$ subject to the following relations:

\[
\begin{aligned}[c]
I \circ I &= I \\
I \otimes I \circ X  &= X = X \circ I \otimes I \\
 X \circ X &= I \otimes I \\
\cap \circ X &= \varepsilon \cap \\
\cap \otimes I \circ I \otimes \cup &= I\\
\cap \otimes I\circ I \otimes X &= I \otimes \cap  \circ X \otimes I \\
I \otimes I \otimes \cup \circ \cup &= \varepsilon\left( \cup \otimes I \otimes I \circ \cup\right) \\
\cap \otimes I \otimes I \circ I \otimes I \otimes \cup &= \varepsilon\left(\cup \circ \cap \right)
\end{aligned}
\hspace{.25in}
\begin{aligned}[c]
I \otimes I \circ \cup &= \cup \\
\cap \circ I \otimes I &= \cap \\
 X  \circ \cup &= \cup \\
\cap \circ \cup &= \delta\\
I \otimes \cap \circ \cup \otimes I &= \varepsilon I\\
 I \otimes X \circ \cup \otimes I &= X \otimes I  \circ I \otimes \cup \\
\cap \circ I \otimes I \otimes \cap &= \varepsilon\left(\cap \circ \cap \otimes I \otimes I\right) \\
X \otimes I \circ I \otimes X \circ X \otimes I &= I \otimes X \circ X \otimes I \circ I \otimes X 
\end{aligned}
\]

Furthermore, by declaring $I$ and $X$ to be even and $\cup$ and $\cap$ to be odd, we have a $\Z_{2}$-grading on the $\Hom$-spaces which coincides with the one given in \cref{R:grading}.
\end{theorem}

\begin{proof}  It is not difficult to see that the given elements generate; that is, that every morphism is a linear combination of elements obtained by the various possible compositions and tensor products of $I$, $X$, $\cup$, and $\cap$.  It is also not difficult to verify that the listed relations are satisfied in $\BB (\delta, \varepsilon)$ by drawing the corresponding pictures and using the rules given in \cref{S:MarkedBrauerDiagrams}.  The most difficult point of the proof is to verify that the given list is a complete set of relations.  That is, if a morphism is written in two ways using the given generators that it is possible to rewrite one into the other only using linear combinations of equations derived from the listed relations via tensor products and compositions.  For this it suffices to note that the proof of \cite[Theorem 2.6]{LZ1} applies verbatim to our situation as long as we take care to account for the powers of $\varepsilon$ which appear.
\end{proof}

\subsection{}\label{SS:duality}  We now introduce several maps which will also be needed in the next section and in the proof of \cref{T:faithful}.  
Given nonnegative integers $a,b,r,s$ with $a \leq r$ and $b \leq s$, we have the following linear maps on the spaces of marked Brauer diagrams: 
\begin{equation}\label{E:Jmaps}
\begin{aligned} 
\rightupJ_{b}&:B_{r,s}(\delta, \varepsilon) \to B_{r+b,s-b}(\delta, \varepsilon) \\
\leftupJ_{b}&:B_{r,s}(\delta, \varepsilon) \to B_{r+b,s-b}(\delta, \varepsilon) \\
\rightdownJ_{a}&:B_{r,s}(\delta, \varepsilon) \to B_{r-a,s+a}(\delta, \varepsilon)  \\
\leftdownJ_{a}&:B_{r,s}(\delta, \varepsilon) \to B_{r-a,s+a}(\delta, \varepsilon) 
\end{aligned}
\end{equation}
Each map is given by tensor products and compositions of marked Brauer diagrams, and scaling by some power of $\varepsilon$.   We define them on marked Brauer diagrams as follows:

\begin{figure}[h]
\centering
\includegraphics[scale=1]{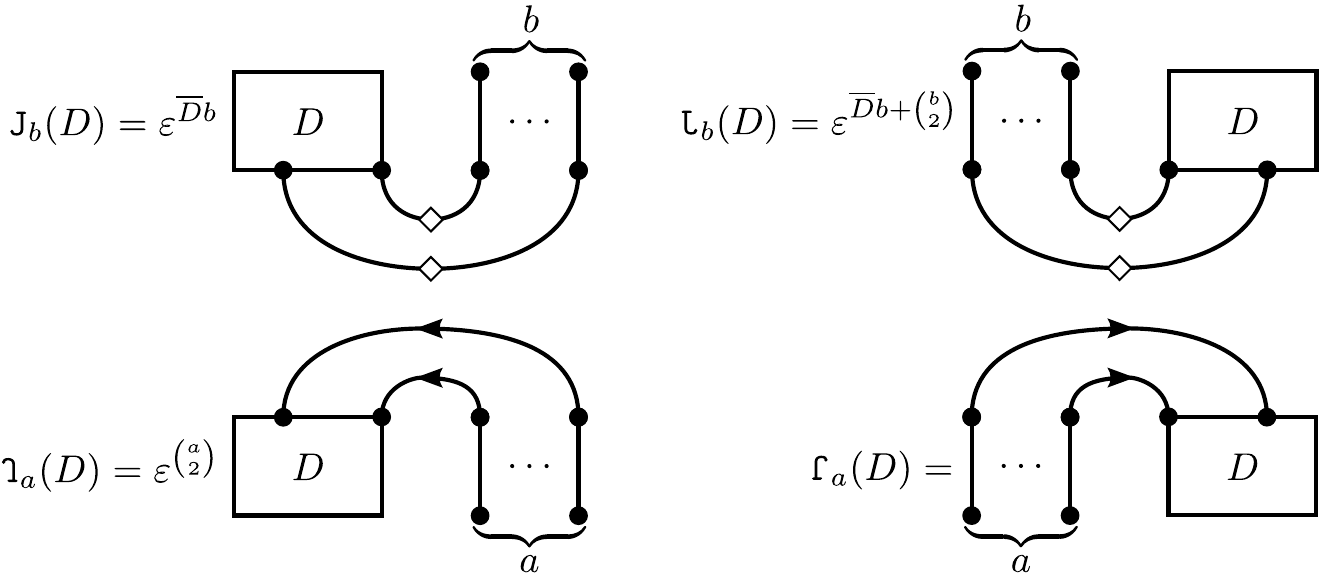}.
\end{figure}

We note that each map is homogeneous in the $\Z_{2}$-grading and with degree given by $a$ or $b$, as the case may be, reduced modulo two.   A straightforward calculation proves the following lemma and explains the choice of scalings by $\varepsilon$.

\begin{lemma}\label{L:Jinverses} Given nonnegative integers $a,b,r,s$ with $a \leq r$ and $b \leq s$, we have the following 
\begin{equation*}
\rightupJ_{a} \circ \rightdownJ_{a} =\rightdownJ_{b} \circ \rightupJ_{b} = \leftupJ_{a} \circ \leftdownJ_{a} = \leftdownJ_{b} \circ \leftupJ_{b} = \Id_{B_{r,s}(\delta, \varepsilon)}.
\end{equation*}
In particular, the maps defined in \cref{E:Jmaps} are homogeneous superspace isomorphisms. 
\end{lemma}

Using these maps we can define a superspace isomorphism which we call the \emph{transpose},
\[
B_{r,s}(\delta, \varepsilon) \to B_{s,r}(\delta, \varepsilon),
\] and is denoted by $D \mapsto D'$.  It is given by 
\begin{equation}\label{E:transposemap}
D' =\varepsilon^{\binom{r}{2}}\left( \leftdownJ_{r} \circ \rightupJ_{s}\right)(D).
\end{equation}
A direct calculation shows that on the generators given in \cref{T:generatorsandrelations} we have $I'=I$, $X'=\varepsilon X$, $\cup'=\varepsilon\cap$, and $\cap'=\cup$.  Using the fact that in the non-trivial case $r+s$ must be even, a direct calculation also verifies that the transpose is a graded anti-homomorphism with respect to the composition operation.  That is, we have the following result.


\begin{lemma}\label{L:transpose}  If $x \in B_{r,s}(\delta, \varepsilon)$ and $y \in  B_{s,t}(\delta, \varepsilon)$ are homogeneous elements, then 
\[
(x \circ y)' = \varepsilon^{\bar{x}\cdot \bar{y}}y' \circ x'.
\]

\end{lemma}

\subsection{}\label{SS:dualityII}  The marked Brauer category admits several endofunctors which arise from natural operations on marked Brauer diagrams.  We could use \cref{T:generatorsandrelations} to define the functors, but it is also straightforward to define them in general and check directly that they define endofunctors in each case.  We do this in the following proposition and leave the verification to the reader, noting that \cref{L:transpose} is used to check the second functor. In the proposition we write $w_{r} \in B_{r,r}(\delta, \varepsilon)$ for the diagram given in Figure~2.  Note that this element is invertible in $B_{r,r}(\delta, \varepsilon)$ under the composition operation.
\pagebreak

\begin{figure}[h]
\centering
\includegraphics{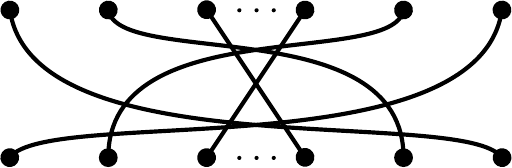}
\caption{The element $w_{r} \in B_{r,r}(\delta, \varepsilon)$.}
\end{figure}

\begin{prop}\label{P:endofunctors} The marked Brauer category admits the following endofunctors:
\begin{enumerate}
\item There is a covariant endofunctor on $\BB (\delta, \varepsilon)$ given on objects by $[a] \mapsto [a]$ and on $x \in B_{r,s}(\delta, \varepsilon)$ by $x \mapsto w_{r} \circ x \circ w_{s}^{-1}$.
\item There is a contravariant endofunctor on $\BB (\delta, \varepsilon)$ given on objects by $[a] \mapsto [a]$ and on $x \in B_{r,s}(\delta, \varepsilon)$ by $x \mapsto x'$, where $x'$ is the transpose defined in the previous section. 
\item There is a contravariant endofunctor on $\BB (\delta, \varepsilon)$ given on objects by $[a] \mapsto [a]$ and on $x \in B_{r,s}(\delta, \varepsilon)$ by  $x \mapsto w_{r} \circ x' \circ w_{s}^{-1}$.
\end{enumerate}

\end{prop}

The first corresponds to reflecting diagrams across the vertical line through the center of the diagram, the second corresponds to rotation of diagrams by $180$ degrees, and the third to reflecting diagrams across the equator.

\section{The Marked Brauer Algebra}\label{S:MarkedBrauerAlgebra}  

\subsection{}  We define the \emph{marked Brauer algebra} to be
\[
B_{r}(\delta, \varepsilon) = \End_{\BB (\delta, \varepsilon)}([r],[r]) = B_{r,r}(\delta, \varepsilon).
\]  This is an associative algebra under the composition operation. We first provide a presentation of the marked Brauer algebra by generators and relations.  For $i=1, \dotsc , r-1$, let $e_{i}, s_{i} \in  B_{r}(\delta, \varepsilon)$ be given by 
\begin{align*}
e_{i} & = \left(I^{\otimes i-1} \otimes \cup \otimes I^{\otimes r-i-1} \right) \circ \left(I^{\otimes i-1} \otimes \cap \otimes I^{\otimes r-i-1} \right) \\
s_{i} &= I^{\otimes i-1} \otimes X \otimes I^{\otimes r-i-1}.
\end{align*}   That is, in terms of marked Brauer diagrams:

\begin{figure}[h]
\centering
\includegraphics[scale=1]{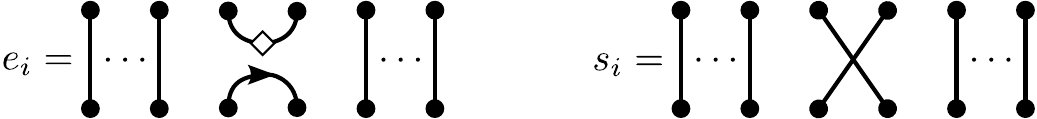}.
\end{figure}

\begin{theorem}\label{C:AlgebraGeneratorsRelations}  As an associative $k$-algebra the marked Brauer algebra $B_{r}(\delta, \varepsilon)$ is isomorphic to the algebra generated by $e_{1}, \dotsc , e_{r-1}$, $s_{1}, \dotsc , s_{r-1}$ subject to the relations:
\[
\begin{aligned}[c]
s_{i}^{2}&=1,\\
e_{i}s_{i}&=\varepsilon e_{i},\\
s_{i}s_{j}&= s_{j}s_{i}, \\
s_{i}s_{i+1}s_{i}&=s_{i+1}s_{i}s_{i+1}, \\
s_{i}e_{j}&=e_{j}s_{i},\\
s_{i}e_{i+1}e_{i}&=\varepsilon s_{i+1}e_{i}, 
\end{aligned}
\hspace{.25in}
\begin{aligned}[c]
e_{i}^{2}&=\delta, \\
s_{i}e_{i}&=e_{i},\\
e_{i}e_{j}&=e_{j}e_{i}, \\
e_{i}e_{i+1}e_{i}&=\varepsilon e_{i}, \\
e_{i+1}e_{i}e_{i+1}&=\varepsilon e_{i+1}, \\
e_{i+1}e_{i}s_{i+1}&=\varepsilon e_{i+1}s_{i}.
\end{aligned}
\]  In the above equations $i$ and $j$ ranges over all possible values from among $1, \dotsc r-1$ subject to $j \neq i, i \pm 1$.

\end{theorem}

\begin{proof}  The result follows from \cref{T:generatorsandrelations} and the observation that each marked Brauer diagram in $B_{r}(\delta, \varepsilon)$  must have an equal number of cups and caps.
\end{proof}

From the definition it is clear that $B_{r}(\delta, 1)$ is isomorphic to the ordinary Brauer algebra $B_{r}(\delta)$.  On the other hand, when $\varepsilon=-1$ we can compare the previous theorem with \cite[Definition 2.2]{Moon} and see that the marked Brauer algebra $B_{r}(0,-1)$ over $\mathbb{C}$ is precisely the algebra $A_{r}$ defined by Moon via generators and relations.  In particular, \cref{L:basis} provides an elementary proof of \cite[Theorem 2.9]{Moon} and the marked Brauer algebra provides a unifying explanation for the similarities between $A_{r}\cong B_{r}(0,-1)$ and $B_{r}(0,1)$ observed by Moon.

The contravariant endofunctors in \cref{P:endofunctors} define involutive anti-homomorphisms (or, for short, anti-involutions) on $B_{r}(\delta, \varepsilon)$.  A direct computation verifies the following lemma.

\begin{lemma}\label{L:involutions}  The marked Brauer algebra $B_{r}(\delta, \varepsilon)$ admits the following two anti-involutions:
\begin{enumerate}
\item The map $x \mapsto x'$.  On generators it is given by $e_{i}\mapsto e_{r-i}$ and $s_{i}\mapsto \varepsilon s_{r-i}$.
\item The map $x \mapsto x^{\bullet}:=w_{r}\circ x' \circ w_{r}^{-1}$.  On generators it is given by $e_{i}\mapsto \varepsilon e_{i}$ and $s_{i}\mapsto \varepsilon s_{i}$.
\end{enumerate}
\end{lemma}

These maps admit natural diagrammatic descriptions.   Up to a power of $\varepsilon$, on marked Brauer diagrams the first anti-involution is rotation by $180$ degrees and the second is reflection across the equator.

\subsection{}\label{SS:iteratedinflation}  We now study the marked Brauer algebra using techniques introduced by Koenig-Xi to study the ordinary Brauer algebra in \cite{KX1,KX2}. First note that we may identify the group algebra of the symmetric group $k\Sigma_{r}$ as the subalgebra of $B_{r}(\delta, \varepsilon)$ spanned by marked Brauer diagrams which consist of only through strings.

For $t=0, \dotsc , r$, let $\ideal_{t}$ be the subspace of $B_{r}(\delta, \varepsilon)$ spanned by all marked Brauer diagrams with $t$ or fewer through strings.  It is straightforward to verify that each $\ideal_{t}$ is a two-sided ideal of $B_{r}(\delta, \varepsilon)$ and that we have the filtration: 
\begin{equation}\label{E:filtration}
0 \subset \ideal_{a} \subset \ideal_{a+2}  \subset \dotsb \subset \ideal_{r-2} \subset \ideal_{r}=B_{r}(\delta, \varepsilon),
\end{equation}
where $a$ is zero or one depending on the parity of $r$.  

Given an algebra $B$, a vector space $V$, and a bilinear form $\langle-, -\rangle: V\otimes V \to B$, the authors of  \cite{KX2} define the \emph{inflation}\footnote{To be precise, the authors of \cite{KX2} also require a compatible anti-involution as part of their definition of inflation.} of $B$ along $V$ to be the algebra $A=V\otimes V \otimes B$ with multiplication given by 
\[
\left( x_{1} \otimes y_{1} \otimes b_{1}\right)\cdot \left( x_{2} \otimes y_{2} \otimes b_{2}\right) = x_{1}\otimes y_{2} \otimes b_{1}\langle y_{1}, x_{2}\rangle b_{2}.
\]  As discussed in \cite{KX2}, with the appropriate assumptions this defines an associative algebra and this construction is equivalent to the generalized matrix algebras first introduced by Brown \cite{Brown} and is also implicit in \cite{GL}.  Our first goal is to show that each $\ideal_{t}/\ideal_{t-2}$ in \cref{E:filtration} is an inflation.

We would first like to set some notation. Given $0 \leq t \leq r$ with $t=r-2k$ for some $k$, let $E_{r,t}$ be the set of all graphs on $r$ vertices (which we label with $1, \dotsc ,r$) which have precisely $k=(r-t)/2$ edges and each vertex is incident to at most one edge.  Given a standard marked Brauer diagram $D$, let $t_{D}$ be the full subgraph of $D$ obtained by taking the vertices labelled by $1, \dotsc , r$ (i.e.\ the vertices along the top row of $D$) and forgetting all markings. Similarly, let $b_{D}$ be the full subgraph of $D$ obtained by taking the vertices labelled by $1', \dotsc , r'$ (i.e.\ the vertices along the bottom row of $D$) and forgetting all markings.  If $D$ is a standard marked Brauer diagram with $t$ through strings, then $t_{D}, b_{D} \in E_{r,t}$.  And, indeed, $D \mapsto t_{D}$ and $D \mapsto b_{D}$ defines bijections between the set of possible top and bottom rows, respectively, of standard marked Brauer diagrams in $B_{r}(\delta, \varepsilon)$ and the set $E_{r,t}$.

 Finally, given a standard marked Brauer diagram with $t$ through strings, $D$, let $\sigma_{D} \in \Sigma_{t}$ be defined as follows.  First, we ignore all cups and caps in $D$.  We may then relabel the vertices in the top row which are incident to a through string with $1, \dotsc , t$ and label the vertices in the bottom row which are incident to a through string with $1', \dotsc , t'$.  In this way we obtain a diagram $\sigma_{D} \in B_{t}(\delta, \varepsilon)$ consisting only of through strings and, hence, we can view it as an element of $\Sigma_{t}$.  

 For the next lemma we use the convenient convention that $k\Sigma_{0}=k$. 

\begin{theorem}\label{T:inflation}  Fix $0 \leq t \leq r$ with $t=r-2k$ for some non-negative integer $k$. Let $V_{t}$ be the $k$-vector space with basis the elements of $E_{r,t}$. Then $\ideal_{t}/\ideal_{t-2}$ is isomorphic as a $k$-algebra (possibly without unit) to the inflation 
\[
V_{t} \otimes V_{t} \otimes k\Sigma_{t}.
\]  The bilinear form will be given in the proof.
\end{theorem}

\begin{proof}  We first define a linear map $\psi: \ideal_{t}/\ideal_{t-2} \to V_{t}\otimes V_{t} \otimes k\Sigma_{t}$. We note that $\ideal_{t}/\ideal_{t-2}$ has a basis given by the set of all standard marked Brauer diagrams which have precisely $t$ through strings.  Given such a marked Brauer diagram, $D$, define 
\[
\psi(D)= t_{D} \otimes b_{D} \otimes \sigma_{D}.
\] It is straightforward to see that this is a bijection between bases and, hence, is a vector space isomorphism.  

We now construct the bilinear form 
\[
\langle - , - \rangle: V_{t} \otimes V_{t} \to k\Sigma_{t}.
\]   Since $E_{r,t}$ provides a basis for $V_{t}$ it suffices to define $\langle e, f \rangle$ for $e,f \in E_{r,t}$.  To do so, we let $\Gamma$ denote the graph obtained from $e$ and $f$ by identifying the vertices of $e$ and $f$ which share the same label.  There are three types of paths in $\Gamma$.  The first are paths which form a closed loop in $\Gamma$.  The second are paths which do not form a closed loop and where the edges which begin and end the path are both edges in $e$ or are both edges in $f$.  The third are paths which do not form a closed loop and where the edge which begins the path lies in $e$ (or, resp.\ $f$) and the edge which ends the path lies in $f$ (resp.\ $e$).  A path consisting of a single edge is considered to be of the second type and an isolated vertex is considered to be a path of the third type.  Now we define the bilinear form by considering two possible cases:
\begin{itemize}
\item [Case I:] If $\Gamma$ contains any paths of the second type, then $\langle e, f \rangle = 0$.
\item [Case II:] If $\Gamma$ contains only closed loops and paths of the third type, then 
\[
\langle e, f \rangle =  \omega_{e,f},
\] where $\omega_{e,f} \in k\Sigma_{t}$ is defined by the equation 
\[
\psi \left(\psi^{-1}\left(e \otimes e \otimes \Id_{k\Sigma_{t}} \right)\psi^{-1}\left(f \otimes f \otimes \Id_{k\Sigma_{t}} \right) \right) = e \otimes f \otimes \omega_{e,f}.
\]
\end{itemize}

The reader can readily check that $\psi$ is an algebra map. In particular, the powers of $\delta$ and $\varepsilon$ one expects to see from the definition of composition in \cref{SS:TwoOperations} are accounted for in the definition of $\omega_{e,f}$.
\end{proof}

There is a marked version of \cite[Lemma 5.5]{KX1}.  Using this we can conclude that the marked Brauer algebra is an iterated inflation (recall that we don't require a compatible anti-involution).  We now consider the question of anti-involutions on the marked Brauer algebra.  In \cref{L:involutions} we saw that $B_{r}(\delta, \varepsilon)$ admits the anti-involution $x \mapsto x^{\bullet}$ corresponding to reflecting marked Brauer diagrams across the equator.  This anti-involution stabilizes the filtration of $B_{r}(\delta, \varepsilon)$ and, hence, defines an anti-involution on each $\ideal_{t}/\ideal_{t-2}$.  Under the map $\psi$ this corresponds to the anti-involution on $V_{t} \otimes V_{t} \otimes k\Sigma_{t}$ given on basis elements by 
\begin{equation}\label{E:gradedinvolution}
e \otimes f \otimes \sigma \mapsto \varepsilon^{k} f \otimes e \otimes \varepsilon^{\ell(\sigma)}\sigma^{-1},
\end{equation}
where $\ell(\sigma)$ is the length of the permutation $\sigma$.  When $\varepsilon=1$ this shows, as in \cite{KX2}, that we have a compatible anti-involution and that the ordinary Brauer algebra $B_{r}(\delta, 1)$ is a cellular algebra in the sense of \cite{GL}.  However, when $\varepsilon=-1$ this anti-involution does not satisfy the compatibility condition given in \cite[Section 3.1]{KX2}.  Indeed, a direct analysis shows that $B_{2}(0,-1)$ is not a cellular algebra for any possible anti-involution.  Nonetheless, \cref{E:gradedinvolution} can be understood as a graded version of the condition in \cite{KX2} and we expect that many of the techniques of cellular algebras will carry over to this setting.

\subsection{}

Using \cref{T:inflation} we can classify the simple modules for $B_{r}(\delta, \varepsilon)$ following \cite{KX2}.

\begin{theorem}\label{T:simplesoftheMBA}  Let $k$ be a field of characteristic $p \neq 2$ and write $r=2u+t$ with $t$ equal to $0$ or $1$. Then the simple modules of the marked Brauer algebra $B_{r}(\delta, \varepsilon)$ are parameterized by:
\begin{enumerate}
\item  The $p$-regular partitions of $r,r-2,r-4,\dotsc, t$ if $\delta\neq 0$ or if $\delta=0$ and $r$ is odd. 
\item The $p$-regular partitions of $r,r-2,r-4,\dotsc, 2$ if $\delta= 0$ and $r$ is even. 
\end{enumerate}
\end{theorem}

\begin{proof}  We first make the following observation.   Let $A$ be a not necessarily unital finite dimensional $k$-algebra, let $I$ be a two-sided ideal of $A$, and assume $A$ admits an anti-involution which stabilizes $I$.   It is then straightforward to verify that the isomorphism classes of simple modules of $A$ are in bijection with the disjoint union of the isomorphism classes of simple modules for $I$ and $A/I$.  This along with \cref{E:filtration} implies that it suffices to classify the simple modules of each $V_{t} \otimes V_{t} \otimes k\Sigma_{t}$ in \cref{T:inflation}. This is done in \cite[Section 3.2 and Corollary 5.8]{KX2} for the ordinary Brauer algebra $B_{r}(\delta, 1)$ and the same analysis applies to the marked Brauer algebra in general. 
\end{proof}

\section{Representations of Lie Superalgebras}\label{S:LSAs}

\subsection{}\label{SS:LSAprelim} We now introduce our original motivation for considering the marked Brauer category.  
Recall from \cref{SS:prelims} that tensor products of superspaces and linear maps between superspaces have a natural grading.   In particular, we say a bilinear form on a superspace $V$ is even or odd if the corresponding linear map $V \otimes V \to k$ is even or odd, respectively.  A bilinear form on a superspace is said to be \emph{symmetric} if it is symmetric in the graded sense; that is, if $(x,y) = (-1)^{\bar{x} \cdot \bar{y}} (y,x)$ for all homogeneous $x,y \in V$.  Note that an anti-symmetric bilinear form on a superspace $V$ corresponds to a symmetric bilinear form on the superspace obtained by switching the grading on $V$.  Consequently there is no loss in only considering symmetric bilinear forms in what follows.  As usual a bilinear form is said to be nondegenerate if for every nonzero $x \in V$ there exists a $y \in V$ such that $(x,y) \neq 0$.

Throughout we will assume we have a fixed superspace $V$ of dimension $m|n$ with a homogeneous bilinear form, $b=(-,-)$, which is nondegenerate and supersymmetric. Given such a bilinear form we define a Lie superalgebra $\fg (V,b) \subseteq \gl (V)= \End_{k}(V)$ consisting of all linear maps which preserve the bilinear form:

\[
\fg(V,b) =\left\{A \in \gl (V) \mid (Ax,y) + (-1)^{\bar{A} \bar{x}}(x,Ay)=0 \right\}.
\]  Note that the defining condition for $\fg (V,b)$ is given only for homogenous elements of $\gl (V)$ and $V$; the general condition can be obtained using linearity.

On the one hand, when the bilinear form is even $\fg(V,b)$ is isomorphic to the orthosymplectic Lie algebra $\mathfrak{osp}(V)$.  In particular, if the dimension of $V$ is $m|0$, then the bilinear form is a nondegenerate symmetric bilinear form in the non-super sense and $\fg(V,b)$ is isomorphic to the orthogonal Lie algebra.  If the dimension of $V$ is $0|n$, then the bilinear form is a nondegenerate skew-symmetric bilinear form in the non-super sense and $\fg(V,b)$ is the symplectic Lie algebra. On the other hand, when the bilinear form is odd the dimension of $V$ is necessarily $n|n$ and $\fg (V,b)$ is isomorphic to the Lie superalgebra $\fp (n)$.  Matrix realizations of these Lie superalgebras can be found in, for example, \cite{Kac}.

Set $\fg =\fg (V,b)$.  A $\fg$-module is a superspace $W$ with an action of $\fg$ which respects the grading in the sense that $\fg_{r}.M_{s} \subseteq M_{r+s}$ for all $r,s \in \Z_{2}$ and which satisfies graded versions of the axioms for a Lie algebra module.   See \cite{Kac} for the basics of Lie superalgebras and their representations.  In particular, $V$ itself is naturally a $\fg$-module.

We choose to write our $\fg$-module homomorphisms on the right.  Consequently,  
\[
\Hom_{\fg}(W,U) = \left\{f \in \Hom_{k}(W,U) \mid (x.w)f = x.(w)f \text{ for all $x \in \fg$ and $w \in W$}  \right\}.
\]  These $\Hom$-spaces are naturally $\Z_{2}$-graded by declaring $f$ to be of degree $r \in \Z_{2}$ if $(W_{s})f \subseteq W_{r+s}$ for all $s \in \Z_{2}$.
The universal enveloping superalgebra of $\fg$ admits a coproduct so we may consider tensor products of finite-dimensional $\fg$-modules. 

Let $\TT_{\fg}(V)$ be the full subcategory of $\fg$-modules consisting of all objects of the form $V^{\otimes a}$ for some $a=0, 1, 2, 3, \dotsc$ (where $V^{\otimes 0} = k$ by convention).  We let $\widehat{\TT}_{\fg}(V)$ be the full subcategory of $\fg$-modules consisting of all objects which are isomorphic to a direct summand of a finite direct sum of objects from $\TT_{\fg}(V)$.  Both categories are easily seen to be symmetric tensor categories with $\widehat{\TT}_{\fg}(V)$ also closed under retracts and finite direct sums.

\subsection{} Given a superspace $V$ of dimension $m|n$ which has a nondegenerate, symmetric, homogeneous bilinear form $b=(-,-)$, set $\varepsilon = (-1)^{\bar{b}}$ and $\delta=m-n$, the superdimension of $V$.  Note that if the bilinear form is odd, then $\varepsilon=-1$ and the non-degeneracy of the bilinear form implies that $m=n$ and $\delta=0$.

Fix a homogeneous basis $v_{1}, \dotsc , v_{m+n}$ for $V$ and let $v_{1}^{*}, \dotsc , v_{m+n}^{*}$ be the homogeneous basis for $V$ defined by $(v_{i}, v_{j}^{*})= \delta_{i,j}$.  We can then define 
\[
\sum_{i=1}^{m+n} (-1)^{\bar{v}_{i}} v_{i} \otimes v_{i}^{*} \in  V \otimes V.
\]  It is straightforward to verify that this vector is independent of the choice of homogeneous basis and that $\fg$ acts trivially on it.  Consequently we can define a $\fg$-module homomorphism $k \xrightarrow{c} V \otimes V$ by 
\begin{equation}\label{E:trivialvector}
(1)c=\sum_{i=1}^{m+n} (-1)^{\bar{v}_{i}} v_{i} \otimes v_{i}^{*}.
\end{equation}

We also have the $\fg$-module homomorphisms $V \xrightarrow{\Id_{V}} V$, $V \otimes V \xrightarrow{b} k$, and $V \otimes V \xrightarrow{s} V \otimes V$ given on our homogeneous basis by 
\begin{align*}
(v_{i})\Id_{V} &= v_{i},\\
(v_{i}\otimes v_{j})b &= (v_{i}, v_{j}),\\
(v_{i}\otimes v_{j})s &= (-1)^{\bar{v}_{i} \bar{v}_{j}} v_{j} \otimes v_{i}.
\end{align*}

\begin{theorem}  Given a superspace $V$ of dimension $m|n$ which has a nondegenerate, symmetric, homogeneous bilinear form $b=(-,-)$, set $\varepsilon = (-1)^{\bar{b}}$ and $\delta=m-n$. Then there exists a unique covariant functor of $k$-linear tensor categories 
\[
F: \BB(\delta, \varepsilon) \to T_{\fg}(V)
\] given on objects by 
\[
F([a]) = V^{\otimes a} 
\] and on the generators $I, X, \cup$, and $\cap$ by: 
\[
F(I) = \Id_{V}, \; \; 
F(X) = s, \; \; 
F(\cup) = b, \; \; 
F(\cap) = c.
\] 
 Furthermore, this functor preserves the $\Z_{2}$-grading of the $\Hom$-spaces.
\end{theorem}

\begin{proof} By \cref{T:generatorsandrelations} it suffices to verify that the relations given in \cref{T:generatorsandrelations} are satisfied by the maps in $T_{\fg}(V)$ given by their images under the putative functor.   We leave this elementary verification to the reader.  It is useful for the calculations which involve the map $c$ to note that $\bar{v}_{i}^{*}= \bar{v}_{i} + \bar{b}$ and that we can construct the vector in \cref{E:trivialvector} by instead using the homogeneous basis $v_{1}^{*}, \dotsc , v_{m+n}^{*}$.  When we do so the vector dual to $v_{i}^{*}$ with respect to the bilinear form is $(-1)^{\bar{v}_{i}\bar {v}_{i}^{*}}v_{i}$. 
\end{proof}

 More generally, if we let $\widehat{\BB}(\delta, \varepsilon)$ denote the Karoubi envelope of the additive envelope of $\BB(\delta, \varepsilon)$, then the universal property of this category implies that $F$ induces a unique functor 
\[
\widehat{F}: \widehat{\BB}(\delta, \varepsilon) \to \widehat{\TT}_{\fg}(V).
\]

\subsection{} Recall that the maps $\rightupJ_{b}$,  $\leftupJ_{b}$,  $\rightdownJ_{a}$, and  $\leftdownJ_{a}$ are given by compositions and tensor products of maps (i.e. marked Brauer diagrams) in $\BB (\delta, \varepsilon)$. Since $F$ is a tensor functor we can define maps between the corresponding $\Hom$-spaces in $T_{\fg}(V)$ by applying $F$ to the constituent marked Brauer diagrams to obtain module homomorphisms and then applying the same sequence of compositions and tensor products.  We call the resulting maps by the same name.  Therefore, given nonnegative integers $a,b,r,s$ with $a \leq r$ and $b \leq s$, we have homogeneous maps of superspaces:
\begin{align}\label{E:JHommaps}
\rightupJ_{b}&:\Hom_{\fg}(V^{\otimes r}, V^{\otimes s})\to \Hom_{\fg}(V^{\otimes r+b}, V^{\otimes s-b}) \notag \\
\leftupJ_{b}&:\Hom_{\fg}(V^{\otimes r}, V^{\otimes s})\to \Hom_{\fg}(V^{\otimes r+b}, V^{\otimes s-b}) \\
\rightdownJ_{a}&:\Hom_{\fg}(V^{\otimes r}, V^{\otimes s}) \to \Hom_{\fg}(V^{\otimes r-a}, V^{\otimes s+a}) \notag \\
\leftdownJ_{a}&:\Hom_{\fg}(V^{\otimes r}, V^{\otimes s}) \to \Hom_{\fg}(V^{\otimes r-a}, V^{\otimes s+a}).  \notag
\end{align}  
It is immediate that if we let $\mathbb{J}$ stand for any one of $\rightupJ_{b}$, $\leftupJ_{b}$, $\rightdownJ_{a}$, or $\leftdownJ_{a}$, then we have
\[
\mathbb{J} \circ F = F \circ \mathbb{J}.
\] Furthermore, applying the functor $F$ to \cref{L:Jinverses} yields 
\begin{equation*}
\rightupJ_{a} \circ \rightdownJ_{a} =\rightdownJ_{b} \circ \rightupJ_{b} = \leftupJ_{a} \circ \leftdownJ_{a} = \leftdownJ_{b} \circ \leftupJ_{b} = \Id_{\Hom_{\fg}(V^{\otimes r}, V^{\otimes s})}.
\end{equation*}  In particular, each map is a homogeneous superspace isomorphism.

\subsection{}  Given nonnegative integers $r$ and $s$ the functor $F$ provides a grading preserving linear map
\begin{equation}\label{E:FonHoms}
F:B_{r,s}(\delta, \varepsilon) \to \Hom_{\fg}\left(V^{\otimes r}, V^{\otimes s} \right).
\end{equation}
In order to use the marked Brauer category to study representations of $\fg$ it is useful to know when this map is injective and surjective; that is, when the functor is faithful and full.  This question is still open in general, but various special cases are known (e.g. \cite{stroppel, HX, LZ0, LZ2, LZ3, Moon}. Addressing this question in full generality is beyond the scope of this paper.  For the remainder of this section we assume $k=\mathbb{C}$  and satisfy ourselves with the following results.


\begin{theorem}\label{T:faithful} Let $r$ and $s$ be nonnegative integers and let $V$ be a superspace of dimension $m|n$ with a homogeneous nondegenerate bilinear form.  We then have the following:
\begin{enumerate}
\item  If the bilinear form is even, then the map in \cref{E:FonHoms} is injective whenever $r+s \leq m + n/2$. 
\item   If the bilinear form is odd, then the map in \cref{E:FonHoms} is injective whenever $r+s \leq m+n$.  

\end{enumerate}
\end{theorem}

\begin{proof} 
Using the maps given in \cref{E:JHommaps} we see it suffices to consider the case when $r=2t$ and $s=0$.  Recall that we can identify the group algebra of the symmetric group on $r$ letters as a subalgebra of $B_{r,r}(\delta, \varepsilon)$.   For any $\sigma \in \Sigma_{r}$ we then have a vector space automorphism  of $B_{r,0}(\delta, \varepsilon)$ given by $x \mapsto \sigma \circ x$. Furthermore, functoriality implies that this map stabilizes the kernel of $F$.  

Now let $z \in B_{r,0}(\delta, \varepsilon)$ be nonzero.  If we write $z = \sum_{D} \alpha_{D}D$, where the sum is over our basis of standard marked Brauer diagrams, then $\alpha_{D}$ is nonzero for some $D$.   The discussion in the previous paragraph  implies we may compose with an appropriate $\sigma \in \Sigma_{r}$ to assume without loss of generality that in fact $\alpha_{D_{0}} \neq 0$ where $D_{0}:=\cup^{\otimes t}$.

We first consider the case when the bilinear form is even.  In this case it restricts to a symmetric nondegenerate bilinear form on $V_{\0}$ and an anti-symmetric nondegenerate bilinear form on $V_{\1}$.  Consequently, we can choose an orthonormal basis $v_{1}, \dotsc , v_{m}$ for $V_{\0}$ and a basis $w_{1}, \dotsc , w_{p}, w^{*}_{1}, \dotsc , w^{*}_{p}$ for $V_{\1}$ such that $(w_{i}, w_{j})=(w^{*}_{i}, w^{*}_{j})=0$ and $(w_{i}, w^{*}_{j})=\delta_{i,j}$ for all $1 \leq i,j \leq p$.  We then have a basis for $V^{\otimes r}$ consisting of the pure tensors of these vectors.  In particular, if we set 
\begin{equation*}
u_{k}= \begin{cases} v_{k} \otimes v_{k}, & 1 \leq k \leq m; \\
                     w_{m-k} \otimes w^{*}_{m-k}, & m+1 \leq k \leq m+p,
\end{cases}
\end{equation*} then $u_{1} \otimes \dotsb \otimes u_{t}$ is a basis element of $V^{\otimes r}$.  Now we evaluate $F(z)$ on this vector and obtain 
\[
\left(u_{1} \otimes \dotsb \otimes u_{t} \right)F(z)  = \sum_{D} \alpha_{D}\left(u_{1} \otimes \dotsb \otimes u_{t} \right)F(D).
\]  However a straightforward calculation shows that 

\begin{equation*}
\left(u_{1} \otimes \dotsb \otimes u_{t} \right)F(D) = \begin{cases} \pm 1, & \text{ if $D = D_{0}$}; \\
                                                                      0,     & \text{otherwise}.
\end{cases}
\end{equation*}
 Consequently $F(z)$ is nonzero and, hence, $F$ is injective.

When the bilinear form is odd we argue just as above except that in this case $m=n$ and we choose a basis $v_{1}, \dotsc , v_{n}$ for $V_{\0}$ and $v^{*}_{1}, \dotsc , v^{*}_{n}$ such that $(v_{i}, v_{j})=(v^{*}_{i}, v^{*}_{j})=0$ and $(v_{i}, v^{*}_{j}) = \delta_{i,j}$ for $1 \leq i, j \leq n$.  We then set $u_{k}=v_{k}\otimes v^{*}_{k}$ and evaluate $F(z)$ on the vector $u_{1} \otimes \dotsb \otimes u_{t}$ as before.
\end{proof}

While the proof given above is uniform in its approach, it does not give the best possible results.  A stronger bound can be found in \cite[Corollary 5.5]{LZ3} when the bilinear form is even.  When the bilinear form is odd our bound agrees with \cite[Theorem 4.1]{Moon}.  


\begin{theorem}\label{T:full} Let $r$ and $s$ be nonnegative integers and let $V$ be a superspace of dimension $m|n$ with a homogeneous nondegenerate bilinear form.  We then have the following:
\begin{enumerate}
\item  If the bilinear form is even, then the map in \cref{E:FonHoms} is surjective whenever $r+s < m(n+1)/2$. 
\item   If the bilinear form is odd, then the map in \cref{E:FonHoms} is surjective whenever $r+s \leq m+n$.  
\end{enumerate}
\end{theorem}

\begin{proof} In the case when the bilinear form is even we can use the maps given in \cref{E:JHommaps} to assume that $r=0$. The result is then the main theorem of \cite{LZ2} along with the discussion in Section~7.3 of \emph{loc. cit.}\  (since we are using the Lie superalgebra rather than the supergroup).  Alternatively, we could invoke \cite[Theorem~C]{stroppel}.

In the case when the bilinear form is odd, we first prove that 
\[
\Hom_{\fg}\left(V^{\otimes r}, V^{\otimes s} \right)=0
\] whenever $r+s$ is odd.   Using the maps given in \cref{E:JHommaps} we see it suffices to consider the case when $r=0$.  Let $\lambda_{1}, \dotsc , \lambda_{n}$ be elements of $\mathbb{C}$ which are linearly independent over $\mathbb{Q}$.   Using the homogeneous basis $v_{1}, \dotsc , v_{n}, v_{1}^{*}, \dotsc , v_{n}^{*}$ fixed in the proof of the previous theorem we define a linear map $A: V \to V$ by $v_{i} \mapsto \lambda_{i}v_{i}$ and $v_{i}^{*} \mapsto -\lambda_{i}v_{i}^{*}$ for all $i=1, \dotsc , n$.  Since this map preserves the bilinear form it defines an (even) element of $\fg$.  The existence of a nontrivial $\fg$-linear map from $V^{\otimes 0}=\mathbb{C}$ to $V^{\otimes s}$ implies that there exists a vector in $V^{\otimes s}$ on which $A$ acts as zero.  The pure tensors in our basis form a basis of eigenvectors for the action of $A$ on $V^{\otimes s}$ and the linear independence of $\lambda_{1}, \dotsc , \lambda_{n}$ implies that the eigenvalue zero cannot occur unless $s$ is even. 

Now that we may assume $r+s$ is even we can use the maps given in \cref{E:JHommaps} to further assume $r=s$.  Surjectivity then follows from \cite[Theorem 4.5]{Moon} since $B_{r,r}(0, -1)$ identifies with Moon's algebra $A_{r}$ and $F$ is exactly the map $\Psi$ defined in \emph{loc. cit}.  
\end{proof}

\subsection{}

Combining \cref{T:faithful} and \cref{T:full} we see that when $k=\mathbb{C}$ we have
\begin{equation}\label{E:algebraisom}
B_{r}(\delta, \varepsilon) \cong \End_{\fg}\left(V^{\otimes r} \right)
\end{equation}
whenever $m$ or $n$ is sufficiently large compared to $r$.  In this case, since the simple modules for $B_{r}(\delta, \varepsilon)$ are in bijection with its conjugacy classes of primitive idempotents and these, in turn, are in bijection with the indecomposable summands of $V^{\otimes r}$ as a $\fg$-module  by \cref{E:algebraisom}, we immediately have the following result. We also note that when the characteristic of $k$ is $p=0$, the condition of $p$-regularity is trivially satisfied by all partitions.   

\begin{theorem}\label{T:indecomposablesummands}  Let $k=\mathbb{C}$ and let $V$ be a superspace of dimension $m|n$. Let $b = (-,-)$ be a nondegenerate homogeneous bilinear form on $V$ with invariant Lie superalgebra $\fg$.  Let $\delta=m-n$ and $\varepsilon=(-1)^{\bar{b}}$.

Then whenever $m$ or $n$ are sufficiently large compared to $r$, the indecomposable summands of $V^{\otimes r}$ as a $\fg$-module are parameterized by the partitions given in \cref{T:simplesoftheMBA}.  Furthermore, for a given partition the number of times the corresponding indecomposable summand of $V^{\otimes r}$ occurs is equal to the dimension of the corresponding simple module for $B_{r}(\delta, \varepsilon)$.
\end{theorem}

When $r=2$ and $r=3$ Moon uses careful calculations using $A_{r}\cong B_{r}(0,-1)$ to obtain an explicit decomposition of $V^{\otimes r}$ into indecomposable summands \cite[Section 6]{Moon}.  It is not difficult to check that the above result agrees with Moon's calculations in these cases.

\linespread{1}

\end{document}